\documentclass[12pt,a4paper]{article}
\usepackage{amsthm}
\usepackage[cp1251]{inputenc}   
\usepackage[english, russian]{babel}

\usepackage{amsmath}
\usepackage{amsfonts,amssymb}

\usepackage{graphicx}                 

\usepackage{psfrag}                    
\usepackage{colordvi}

\usepackage{cite}               
\usepackage{srcltx}            
\usepackage{url}
\usepackage{amsxtra}
\usepackage{algorithm}
\usepackage{algorithmic}
\usepackage{caption}
\usepackage{makeidx}
\usepackage{multirow}

\usepackage{mathtools}
\usepackage{algorithm}
\DeclarePairedDelimiter\bracket{(}{)}
\newcommand{\br}[1]{\bracket*{#1}}

\DeclarePairedDelimiter\modbracket{|}{|}
\newcommand{\mbr}[1]{\modbracket*{#1}}
\newtheorem{definition}{Определение}
\newtheorem{example}{Пример}
\newtheorem{remark}{Замечание}
\newtheorem{lemma}{Лемма}
\newtheorem{theorem}{Теорема}
\newtheorem{corollary}{Следствие}

\begin{document}
\begin{center}
{\bf АНАЛОГ КВАДРАТИЧНОЙ ИНТЕРПОЛЯЦИИ ДЛЯ СПЕЦИАЛЬНОГО КЛАССА НЕГЛАДКИХ ФУНКЦИОНАЛОВ И ОДНО ЕГО ПРИЛОЖЕНИЯ К ОЦЕНКАМ СКОРОСТИ СХОДИМОСТИ АДАПТИВНОГО ЗЕРКАЛЬНОГО СПУСКА ДЛЯ УСЛОВНЫХ ЗАДАЧ ОПТИМИЗАЦИИ
\\[2ex]
Ф.~С.~Стонякин}\\[3ex]
\end{center}

 \bigskip

\section{Введение}

Многие численные методы оптимизации основаны на идее удачной аппроксимации оптимизируемого функционала некоторым функционалом стандартного типа. Например, для гладкого целевого функционала $f:\mathbb{R}^n\rightarrow\mathbb{R}$ с липшицевым градиентом
\begin{equation}\label{eq1}
\|\nabla f(x)-\nabla f(y)\|_*\leqslant L\|x-y\|\;\forall x,y\in Q
\end{equation}
на области определения $Q\subset\mathbb{R}^n$ хорошо известно неравенство
\begin{equation}\label{eq2}
|f(y)-f(x)-\langle\nabla f(x),y-x\rangle|\leqslant\frac{L\|y-x\|^2}{2}
\end{equation}
для всяких $x$ и $y$ из $Q$, то есть
\begin{equation}\label{eq3}
\varphi_2(x,y)\leqslant f(y)\leqslant\varphi_1(x,y)\;\forall x,y\in Q,
\end{equation}
где $$\varphi_1(x,y)=f(x)+\langle\nabla f(x),y-x\rangle+\frac{L\|y-x\|^2}{2}, \text{ а }$$
$$\varphi_2(x,y)=f(x)+\langle\nabla f(x),y-x\rangle+\frac{L\|y-x\|^2}{2}.$$

Как известно \cite{bib_Nesterov}, указанные неравенства \eqref{eq2} -- \eqref{eq3} позволяют обосновать не только глобальную сходимость методов градиентного типа, но и оценивать скорость такой сходимости. Отметим также \cite{bib_Nesterov}, что похожие на \eqref{eq2} -- \eqref{eq3} неравенства можно выписать для негладкого функционала, равного максимуму конечного числа гладких функционалов c липшицевым градиентом.

Неравенства \eqref{eq2} -- \eqref{eq3} широко используются для обоснования скорости сходимости самых разных методов для задач как условной, так и безусловной оптимизации. Например, недавно в \cite{bib_Adaptive} были предложены алгоритмы зеркального спуска как с адаптивным выбором шага, так и с адаптивным критерием остановки. При этом помимо случая липшицевых целевого функционала и функционального ограничения в (\cite{bib_Adaptive}, п. 3.3) на базе идеологии \cite{bib_Nesterov,bib_Nesterov2016} был предложен оптимальный с точки зрения нижних оракульных оценок \cite{nemirovsky1983problem} метод для условных задач выпуклой минимизации с целевыми функционалами, обладающими свойством липшицевости градиента. В частности, в задачах с квадратичными функционалами мы сталкиваемся с ситуацией, когда функционал не удовлетворяет обычному свойству Липшица (или константа Липшица достаточно большая), но градиент удовлетворяет условию Липшица. Для задач такого типа в (\cite{bib_Adaptive}, п. 3.3) был предложен адаптивный алгоритм зеркального спуска. Модификация метода (\cite{bib_Adaptive}, п. 3.3) в случае нескольких ограничений рассмотрена в \cite{bib_Stonyakin}.

Основной результат настоящей статьи (теорема \ref{th1}) --- обоснование возможности построения аналога неравенств стандратной квадратичной интерполяции \eqref{eq2} -- \eqref{eq3} для специального класса негладких квазивыпуклых функционалов с {\it липшицевым субградиентом} (определение \ref{Main_Def}). Идея предлагаемой концепции свойства липшицевости субградиента заключается в том, чтобы описать изменение аппроксимации функционала при условии наличия некоторого (не более, чем счётного) набора точек с особенностями. Основной результат работы (теорема \ref{th1}) показывает, как эти особенности могут приводить к изменению модели функционала, пригодной для построения методов оптимизации. Поскольку локально липшицевы квазивыпуклые функционалы могут не иметь субдифференциала в смысле выпуклого анализа, то для описания дифференциальных свойств мы используем наиболее известное обобщение понятия субдифференциала на невыпуклые функционалы --- субдифференциал Кларка \cite{bib_Clarke}, а под субградиентами всюду далее понимаем элементы (векторы) субдифференциала Кларка как множества. Отметим, что для выпуклых функционалов субдифференциал Кларка совпадает с  обычным субдифференциалом в смысле выпуклого анализа. Построен пример негладкого выпуклого функционала из указанного класса, который может иметь сколь угодно большую константу Липшица при нулевой константе Липшица субградиента (пример \ref{Main_Example}). Как приложение, обоснована оптимальность метода (\cite{bib_Stonyakin}, п. 4) для условных задач с локально липшицевым целевым функционалом и несколькими выпуклыми липшицевыми функциональными ограничениями. Для оценки скорости алгоритма 1 доказан аналог известного утверждения (\cite{bib_Nesterov}, лемма 3.2.1) в классе непрерывных квазивыпуклых локально липшицевых функционалов с использованием субдифференицала Кларка для описания их дифференциальных свойств (теорема \ref{Nesterov_lem}).

Всюду далее будем считать, что $(E,||\cdot||)$~--- конечномерное нормированное векторное пространство и $E^*$~--- сопряженное пространство к $E$ со стандартной нормой:
$$||y||_*=\max\limits_x\{\langle y,x\rangle,||x||\leq1\},$$
где $\langle y,x\rangle$~--- значение линейного непрерывного функционала $y$ в точке $x \in E$, $Q\subset E$~--- замкнутое выпуклое множество.

\section{Об аналоге условия Липшица градиента для специального класса негладких функционалов}

В данном разделе мы покажем, как можно обобщить свойство \eqref{eq1} и оценки \eqref{eq2}--\eqref{eq3} на некоторый класс квазивыпуклых локально липшицевых функционалов $f:Q\rightarrow\mathbb{R}\;(Q\subset\mathbb{R}^n)$, не дифференцируемых на некотором счетном подмножестве $Q_0\subset Q$. Напомним, что функционал $f:Q\rightarrow\mathbb{R}$ квазивыпуклый, если:
\begin{equation}\label{eq19}
f((1-t)x+ty)\leqslant\max\{f(x),f(y)\}\quad\forall t\in[0;1]\quad\forall x,y\in Q,
\end{equation}

Введем класс негладких квазивыпуклых функционалов, допускающих аналоги оценок \eqref{eq2}--\eqref{eq3}. Будем считать $f$ дифференцируемой во всех точках $Q\setminus Q_0$ и полагать, что для произвольного $x\in Q_0$ существует компактный субдифференциал Кларка $\partial_{Cl}f(x)$. Напомним это понятие (\cite{bib_Clarke}, $\S$ 2.2). Пусть $x_{0} \in \mathbb{R}^n$ --- фиксированная точка, и $h \in \mathbb{R}^n$ --- фиксированное направление. Положим
$$
f_{Cl}^{\uparrow} (x_{0}; h) = \limsup\limits_{x^\prime
\rightarrow x_{0},~\alpha \downarrow 0}\frac {1}{\alpha}
\left[f(x^\prime + \alpha h) - f(x^\prime)\right]~.
$$

Величинf $f_{Cl}^{\uparrow} (x_{0}; h)$ называется верхней производной Кларка функционала $f$ в точке $x_{0}$ по направлению $h$. Как известно, функция $f_{Cl}^{\uparrow} (x_{0}; h)$ субаддитивна и положительно однородна по $h$ (\cite{bib_Clarke},
с. 17 -- 18). Это обстоятельство позволяет определить субдифференциал функционала $f$ в точке $x_{0}$ как следующее
множество:
\begin{equation}\label{equv_Clarke}
\partial_{Cl} f(x_{0}) := \{v \in \mathbb{R} \quad |\quad f_{Cl}^{\uparrow} (x_{0}; g)\geqslant vg \quad \forall g \in
\mathbb{R}\}~,
\end{equation}
то есть как субдифференциал выпуклого по $h$ функционала $f_{Cl}^{\uparrow} (x_{0}; h) $ в точке $h = 0$ в смысле выпуклого анализа. Таким образом, по определению
\begin{equation}\label{fedyor_f01}
f_{Cl}^{\uparrow} (x_{0}; h) = \max\limits_{v \in
\partial_{Cl} f(x_{0})} \langle v, h \rangle~.
\end{equation}

Будем говорить, что функционал $f$ субдифференцируем по Кларку в точке $x_0$, если множество $\partial_{Cl} f(x_0)$ непусто и компактно. В частности, если функция $f$ локально липшицева, то она является субдифференцируемой по Кларку в любой точке области определения. Отметим, что для выпуклых функций субдифференциал Кларка совпадает с обычным субдифференциалом в смысле выпуклого анализа \cite{bib_Clarke}. В дальнейших рассуждениях для фиксированных $x,y\in Q$ при $t\in[0;1]$ будем обозначать $y_t:=(1-t)x+ty$.

\begin{definition}\label{Main_Def}
Будем говорить, что квазивыпуклый локально липшицевый функционал $f:Q\rightarrow\mathbb{R}\;(Q\subset\mathbb{R}^n)$ имеет $(\delta,L)$-липшицев субградиент ($f\in C_{L,\delta}^{1,1}(Q)$), если:
\begin{itemize}
\item[(i)] для произвольных $x,y\in Q$ $f$ дифференцируем во всех точках множества $\{y_t\}_{0\leqslant t\leqslant1}$, за исключением последовательности (возможно, конечной)
\begin{equation}\label{eq4}
\{y_{t_k}\}_{k=1}^{\infty}:\;t_1<t_2<t_3<\ldots\text{ и }\lim_{k\rightarrow\infty}t_k=1;
\end{equation}
\item[(ii)] для последовательности точек из \eqref{eq4} существуют конечные субдифференциалы Кларка $\{\partial f(y_{t_k})\}_{k=1}^{\infty}$ и
\begin{equation}\label{eq5}
diam\;\partial_{Cl} f(y_{t_k})=:\delta_k>0,\text{ где }\sum_{k=1}^{+\infty}\delta_k=:\delta<+\infty.
\end{equation}
$$(diam\;\partial_{Cl} f(x)=\max\{\|y-z\|_*\,\mid\,y,z\in\partial_{Cl} f(x)\});$$
\item[(iii)] для произвольных $x,y\in Q$ при условии, что $y_t\in Q\setminus Q_0$ при всяком $t\in(0,1)$ (то есть существует градиент $\nabla f(y_k)$) для некоторой фиксированной константы $L>0$, не зависящей от выбора $x$ и $y$, выполняется неравенство:
\begin{equation}\label{eq6}
\min_{\substack{\hat{\partial}f(x)\in\partial_{Cl} f(x),\\\hat{\partial}f(y)\in\partial_{Cl} f(y)}}\|\hat{\partial}f(x)-\hat{\partial}f(y)\|_*\leqslant L\|x-y\|.
\end{equation}
\end{itemize}
\end{definition}

Ясно, что всякий локально липшицев квазивыпуклый функционал, удовлетворяющий \eqref{eq1}, будет входить в класс $C_{L,\delta}^{1,1}(Q)$ при $\delta=0$. Приведем пример негладкой вещественной выпуклой функции $f\in C_{L,\delta}^{1,1}(Q)$ при $\delta>0$.

\begin{example}\label{Main_Example}
Зафиксируем некоторое $k>0$, величину $\delta>0$ и рассмотрим кусочно-линейную функцию $f:[0;1]\rightarrow\mathbb{R}$ (здесь $Q=[0;1]\subset\mathbb{R}$):
\begin{equation}\label{eq7}
f(x):=kx\text{ при }0\leqslant x\leqslant\frac{1}{2},
\end{equation}
$$
f(x):=\left(k+\sum_{i=1}^{n}\frac{\delta}{2^i}\right)x-\sum_{i=1}^n\frac{\delta}{2^i}\left(1-\frac{1}{2^i}\right)\\
\text{при }1-\frac{1}{2^n}<x\leqslant1-\frac{1}{2^{n+1}},
$$
$$
f(1):=\lim_{x\rightarrow+1}f(x).
$$
В этом случае $$Q_0=\left\{1-\frac{1}{2^n}\right\}_{n=1}^{\infty},$$
$$\partial f(q_n)=\left[k+\sum\limits_{i=1}^{n-1}\frac{\delta}{2^i}; ~ k+\sum\limits_{i=1}^n\frac{\delta}{2^i}\right]$$ при $n>1$ (здесь $\partial f(\cdot)$ --- субдифференциал в смысле выпуклого анализа), $$\partial f(q_1)=\left[k;k+\frac{\delta}{2}\right]$$ (здесь $q_n=1-\frac{1}{2^n}$ при $n=1,2,3,\ldots$). Ясно, что $\partial f(q_n)=\frac{\delta}{2^n}$, то есть верно \eqref{eq5} для введенной величины $\delta>0$. При этом на отрезках $(q_n;q_{n+1})$ и $(0;q_1)$ функция $f$ имеет липшицев градиент с константой $L=0$. Поэтому для функции $f$ из \eqref{eq7} верно $f\in C_{0,\delta}^{1,1}(Q)$.
\end{example}

\begin{remark}
Ясно, что функцию $f$ из \eqref{eq7} нельзя представить в виде максимума конечного набора линейных функций, поскольку $f$ имеет  бесконечное число точек недифференцируемости $f$.
\end{remark}

Сформулируем для введенного класса функционалов $C_{L,\delta}^{1,1}(Q)$ аналог леммы 1.2.3 из \cite{bib_Nesterov}.

\begin{theorem}\label{th1}
Пусть локально липшицев квазивыпуклый функционал $f\in C_{L,\delta}^{1,1}(Q)$. Тогда для произвольных $x,y\in Q$ верно неравенство
\begin{equation}\label{eq8}
|f(y)-f(x)-\langle\hat{\partial}f(x),y-x\rangle|\leqslant\frac{L}{2}\|y-x\|^2+\delta\|y-x\|
\end{equation}
для некоторого субградиента $\hat{\partial}f(x)\in\partial_{Cl} f(x)$.
\begin{proof}
Для произвольных фиксированных $x,y\in Q$ через $y_t$ будем обозначать элемент $ty+(1-t)x$. Тогда при фиксированных $x$ и $y$ одномерная функция $\varphi:[0;1]\rightarrow\mathbb{R}$ ($\varphi(0)=f(x)$ и $\varphi(1)=f(y)$)
\begin{equation}\label{eq20}
\varphi(t)=f(y_t)=f((1-t)x+ty)
\end{equation}
будет квазивыпуклой и для некоторого $\hat{t}\in[0;1]$ отрезки $[0;\hat{t}]$ и $[\hat{t};1]$ будут промежутками (вообще говоря, нестрогой) монотонности функции $\varphi$.

Поскольку для всякой точки $y_t\;(t\in[0;1])$ существует конечный субдифференциал Кларка $\partial_{Cl} f(y_t)$, а также функционал $f$  локально липшицев и квазивыпуклый, то для всех $t\in(0;1)$ существуют конечные левосторонняя и правосторонняя производные:
\begin{equation}\label{eq9}
\begin{split}
\varphi_-'(t)=\lim_{\Delta t\rightarrow-0}\frac{\varphi(t+\Delta t)-\varphi(t)}{\Delta t},\quad
\varphi_+'(t)=\lim_{\Delta t\rightarrow+0}\frac{\varphi(t+\Delta t)-\varphi(t)}{\Delta t}
\end{split}
\end{equation}
и
\begin{equation}\label{eq10}
\varphi_+'(t)=\max_{\hat{\partial}f(y_t)\in\partial_{Cl} f(y_t)}\langle\hat{\partial_{Cl}}f(y_t),y-x\rangle\text{~---}
\end{equation}
производная $f$ по направлению $y-x$ в точке $y_t$. Ясно, что при $y_t\not\in Q_0$ (то есть существует градиент $\nabla f(y_t)$)
\begin{equation}\label{eq11}
\varphi_-'(t)=\varphi_+'(t)=\langle\nabla f(y_t),y-x\rangle.
\end{equation}

Ввиду квазивыпуклости $f$(и $\varphi$) можно полагать, что функция $\varphi$ абсолютно непрерывна и почти всюду дифференцируема в смысле классической меры Лебега, т.е. имеем равенства:
$$f(y)=f(x)+\int_{[0;1]\setminus Q_0}\langle\nabla f(y_t),y-x\rangle\,dt=\varphi(0)+\int_0^1\varphi_+'(t)\,dt,$$
откуда для произвольного субградиента $\hat{\partial}f(x)\in\partial f(x)$ имеем:
\begin{equation}\label{eq12}
f(y)=f(x)+\langle\hat{\partial}f(x),y-x\rangle+\int_0^1\left[\max_{\hat{\partial}f(y_t)\in\partial_{Cl} f(y_t)}\langle\hat{\partial}f(y_t),y-x\rangle-\langle\hat{\partial}f(x),y-x\rangle\right]\,dt=
\end{equation}
$$
=f(x)+\langle\hat{\partial}f(x),y-x\rangle+\int_0^1\langle\hat{\partial}f(y_t)-\hat{\partial}f(x),y-x\rangle\,dt
$$
для набора субградиентов $\{\hat{\partial_{Cl}}f(y_t)\}_{t\in(0;1]}$, на которых достигаются соответствующие максимумы. Если $y_t\in Q_0$, то $y_t=q_k\;(k\geqslant1)$ из определения \ref{eq1} (ii) и тогда
\begin{equation}\label{eq13}
\begin{split}
\varphi_+'(t)-\varphi_-'(t)=\langle\hat{\partial}_1f(y_t)-\hat{\partial}_2f(y_t),y-x\rangle=\\
=\langle\hat{\partial}_1f(q_k)-\hat{\partial}_2f(q_k),y-x\rangle\leqslant\\
\leqslant\|\hat{\partial}_1f(q_k)-\hat{\partial}_2f(q_k)\|_*\cdot\|y-x\|\stackrel{\eqref{eq4}}{\leqslant}\delta_k\|y-x\|\\
\end{split}
\end{equation}
для соответствующих субградиентов (векторов-элементов субдифференциалов Кларка)\\ $\hat{\partial}_{1,2}f(q_k) \in \partial_{Cl} f(q_{k})$. Не уменьшая общности рассуждений, будем считать, что
\begin{equation}\label{eq14}
x,y\in Q_0\subset\{y_t\}_{t\in[0;1]}
\end{equation}
и всякому $q_n$ поставим в соответствие $t_n\in[0;1]:q_n=(1-t_n)x+t_ny$.

Пусть существует последовательность
$$\{t_n\}_{n=1}^{\infty}:0=t_1<t_2<\ldots<1,\;\lim_{n\rightarrow\infty}t_n=1.$$
Тогда $\forall\tau_1,\tau_2\in(t_k;t_{k+1})$ при $k\geqslant1$ верны неравенства:
\begin{equation}\label{eq15}
\|\nabla f(y_{\tau_2})-\nabla f(y_{\tau_1})\|_*\leqslant L|\tau_2-\tau_1|\cdot\|y-x\|,
\end{equation}
\begin{equation}\label{eq16}
\left|\varphi_+'(t_k)-\varphi_-'(t_{k+1})\right|\leqslant\frac{1}{2}L(t_{k+1}-t_k)\cdot\|y-x\|^2.
\end{equation}
Поэтому при выборе в \eqref{eq12} подходящего субградиента $\hat{\partial}f(x)$ будут выполняться соотношения:
$$\mbr{f(y)-f(x)-\langle\hat{\partial}f(x),y-x\rangle}\stackrel{\eqref{eq12}}{=}\mbr{\int_0^1\langle\hat{\partial}f(y_t)-\hat{\partial}f(x),y-x\rangle\,dt}=$$
$$=\mbr{\int_0^1(\varphi_+'(t)-\varphi_+'(0))\,dt}\leqslant\mbr{\int_0^1\langle\hat{g}(y_t)-\hat{\partial}f(x),y-x\rangle\,dt}+\sum_{k=1}^{+\infty}\delta_k\|y-x\|,$$
причем $\forall\tau_1,\tau_2\in[0;1]$
\begin{equation}\label{eq17}
\|\hat{g}(y_{\tau_1})-\hat{g}(y_{\tau_2})\|_*\leqslant L\|y_{\tau_1}-y_{\tau_2}\|,
\end{equation}
откуда
$$\mbr{\int_0^1\langle\hat{g}(y_t)-\hat{\partial}f(x),y-x\rangle\,dt}=\mbr{\int_0^1\langle\hat{g}(y_t)-\hat{g}(x),y-x\rangle\,dt}\leqslant$$
$$\leqslant\int_0^1\mbr{\langle\hat{g}(y_t)-\hat{g}(x),y-x\rangle}\,dt\leqslant\int_0^1\|\hat{g}(y_t)-\hat{g}(x)\|_*\,dt\cdot\|y-x\|\stackrel{\eqref{eq17}}{\leqslant}$$
$$\leqslant L\int_0^1\|y_t-x\|\,dt\cdot\|y-x\|\leqslant L\|y-x\|^2\cdot\int_0^1t\,dt=\frac{L}{2}\|y-x\|^2,$$
то есть
\begin{equation}\label{eq18}
|f(y)-f(x)-\langle\hat{\partial}f(x),y-x\rangle|\leqslant\frac{L}{2}\|y-x\|^2+\delta\|y-x\|,
\end{equation}
что и требовалось.
\end{proof}
\end{theorem}

\begin{corollary}
Если $f\in C_{L,\delta}^{1,1}(Q)$, то для произвольных $x, y \in Q$ верны неравенства:
$$f(y)\leqslant f(x)+\max\|\hat{\partial}f(x)\|_*\cdot\|y-x\|+\delta\|y-x\|+\frac{L}{2}\|y-x\|^2=$$
$$=f(x)+\br{\max\|\hat{\partial}f(x)\|_*+\delta}\|y-x\|+\frac{L}{2}\|y-x\|^2.$$
\end{corollary}

\section{Пример приложения: адаптивный зеркальный спуск для задач минимизации квазивыпуклого целевого функционала рассматриваемого класса гладкости}\label{SectLipGrad}

В качестве приложения покажем возможность получения оценок скорости сходимости для метода (\cite{bib_Stonyakin}, алгоритм 4) в более широком классе целевых функционалов. Напомним, что метод (\cite{bib_Stonyakin}, алгоритм 4) мы рассматривали условных задач выпуклой минимизации с условием липшицевости градиента целевого функционала. Например, квадратичный целевой функционал может не удовлетворять обычному свойству Липшица (или константа Липшица может быть довольно большой), но его градиент удовлетворяет условию Липшица. Метод (\cite{bib_Stonyakin}, алгоритм 4) применим и для более широкого класса уже негладких выпуклых целевых функционалов
\begin{equation}\label{equiv_nonstand1}
f(x)=\max\limits_{1\leqslant i\leqslant m} f_i(x),
\end{equation}
где
\begin{equation}\label{equiv_nonstand2}
f_i(x)=\frac{1}{2}\langle A_ix,x\rangle-\langle b_i,x\rangle+\alpha_i,\;i=1,\ldots,m,
\end{equation}
в случае, когда $A_i$ ($i=1,\ldots,m$)~--- положительно определённые матрицы: $x^TA_ix\geqslant 0\ \forall x \in Q$.

Начнём с постановки рассматриваемых задач условной оптимизации, а также необходимых вспомогательные понятий. Рассмотрим набор выпуклых субдифференцируемых функционалов $g_m:X\rightarrow\mathbb{R}$ для $m = \overline{1, M}$. Также предположим, что все функционалы $g_m$ удовлетворяют условию Липшица с некоторой константой $M_g$:
\begin{equation}\label{equv1}
|g_m(x)-g_m(y)|\leqslant M_g||x-y||\quad \forall x,y\in Q, \quad m = \overline{1, M}.
\end{equation}

Мы рассматриваем следующий тип задач оптимизации квазивыпуклого локально липшицева целевого функционала $f$ с выпуклыми липшицевыми функциональными ограничениями.
\begin{equation}\label{equv2}
 f(x) \rightarrow \min\limits_{x\in Q},
\end{equation}
где
\begin{equation}
\label{problem_statement_g}
    g_m(x) \leqslant 0 \quad \forall m = \overline{1, M}.
\end{equation}

Сделаем предположение о разрешимости задачи \eqref{equv2}--\eqref{problem_statement_g}. Задачи минимизации негладкого функционала c ограничениями возникают в широком классе проблем современной large-scale оптимизации и её приложений \cite{bib_ttd,bib_Shpirko}. Для таких задач имеется множество методов, среди которых можно отметить метод зеркального спуска \cite{beck2003mirror,nemirovsky1983problem}. Отметим, что в случае негладкого целевого функционала или функциональных ограничений естественно использовать субградиентные методы, восходящие к хорошо известным работам \cite{polyak1967general, shor1967generalized}. Метод зеркального спуска возник для безусловных задач в \cite{nemirovskii1979efficient,nemirovsky1983problem} как аналог стандартного субградиентного метода с неевклидовым проектированием. Для условных задач аналог этого метода был предложен в \cite{nemirovsky1983problem} (см. также \cite{beck2010comirror}). Проблема адаптивного выбора шага без использования констант Липшица рассмотрена в \cite{bib_Nemirovski} для задач без ограничений, а также в \cite{beck2010comirror} для задач с функциональными ограничениями.

Отметим, что всюду далее будем под субградинетом квазивыпуклого (локально липшицева) функционала $f$ понимать любой элемент (вектор) субдифференциала Кларка. Для выпуклых функционалов $g_m$ понятие субградиента мы понимаем стандартно.

Для дальнейших рассуждений нам потребуются следующие вспомогательные понятия (см., например, \cite{bib_Nemirovski}), позволяющие оценить качество найденного решения. Для оценки расстояния от текущей точки до решения введём так называемую {\it прокс-функцию} $d : X \rightarrow \mathbb{R}$, обладающую свойством непрерывной дифференцируемости и $1$-сильной выпуклости относительно нормы $\lVert\cdot\rVert$, т.е.
$$\langle \nabla d(x) - \nabla d(y), x-y \rangle \geqslant \lVert x-y \rVert^2 \quad \forall x, y, \in X$$
и предположим, что $\min\limits_{x\in X} d(x) = d(0).$ Будем полагать, что имеется некоторая оценка расстояния от точки старта до искомого решения задачи $x_*$, т.е. существует такая константа $\Theta_0 > 0$, что $d(x_{*}) \leqslant \Theta_0^2,$ где $x_*$~--- точное решение (\ref{equv2})--(\ref{problem_statement_g}). Если имеется множество решений $X_*$, то мы предполагаем, что для константы $\Theta_0$
$$\min\limits_{x_* \in X_*} d(x_*) \leqslant \Theta_0^2.$$
Для всех $x, y\in X$ рассмотрим соответствующую дивергенцию Брэгмана
$$V(x, y) = d(y) - d(x) - \langle \nabla d(x), y-x \rangle.$$

В зависимости от постановки конкретной задачи возможны различные подходы к определению прокс-структуры задачи и соответствующей дивергенции Бргэмана: евклидова, энтропийная и многие другие (см., например, \cite{bib_Nemirovski}). Стандартно определим оператор проектирования
$$\mathrm{Mirr}_x (p) = \arg\min\limits_{u\in Q} \big\{ \langle p, u \rangle + V(x, u) \big\} \; \text{ для всяких  }x\in Q \text{ и }p\in E^*.$$
Сделаем предположение о том, что оператор $\mathrm{Mirr}_x (p)$ легко вычислим.

Напомним одно известное утверждение, которое вытекает из обычного неравенства Коши-Буняковского, а также $2ab \leqslant a^2 + b^2$. Поскольку функциональные ограничения у нас по-прежнему выпуклы, мы рассмотрим также отдельно оценку в выпуклом случае \cite{bib_Nemirovski}.

\begin{lemma}\label{lem1}
Пусть $f:X\rightarrow\mathbb{R}$~--- некоторый функционал. Для произвольного $y \in X$, вектора $p_y \in E^*$ и некоторого $h > 0$ положим $z=Mirr_{y}(h \cdot p_y)$. Тогда для произвольного $x\in Q$
\begin{equation}\label{equv7}
h\langle p_y, y-x\rangle\leqslant\frac{h^2}{2}||p_y||_*^2 + V(y,x) - V(z,x).
\end{equation}
Для выпуклого субдифференцируемого в точке $y$ функционала $f$ предыдущее неравенство для произвольного субградиента $p_y = \nabla f(y)$ примет вид
\begin{equation}\label{equv71}
h\cdot(f(y) - f(x)) \leqslant \langle\nabla f(y), y-x\rangle\leqslant\frac{h^2}{2}||\nabla f(y)||_*^2 + V(y,x) - V(z,x).
\end{equation}
\end{lemma}

Аналогично (\cite{bib_Stonyakin}, алгоритм 4) рассмотрим следующий алгоритм адаптивного зеркального спуска для задач (\ref{equv2})--(\ref{problem_statement_g}). Отметим, что ввиду предположения локальной липшицевости квазивыпуклого целевого функционала все его субградиенты конечны. Сделаем дополнительное предположение об отсутствии точек перегиба, т.е. градиент
$f$ может быть нулевым только в точке $x_*$.

\begin{algorithm}
\caption{Адаптивный зеркальный спуск, квазивыпуклый негладкий целевой функционал, много ограничений}
\label{alg4}
\begin{algorithmic}[1]
\REQUIRE $\varepsilon>0,\Theta_0:\,d(x_*)\leqslant\Theta_0^2$
\STATE $x^0=argmin_{x\in Q}\,d(x)$
\STATE $I=:\emptyset$
\STATE $N\leftarrow0$
\REPEAT
    \IF{$g(x^N)\leqslant\varepsilon$}
        \STATE $h_N\leftarrow\frac{\varepsilon}{||\nabla f(x^N)||_{*}}$
        \STATE $x^{N+1}\leftarrow Mirr_{x^N}(h_N\nabla f(x^N))\;\text{// \emph{"продуктивные шаги"}}$
        \STATE $N\rightarrow I$
    \ELSE
        \STATE // \emph{$(g_{m(N)}(x^N)>\varepsilon)\;\text{для некоторого}\; m(N)\in \{1,\ldots,M\}$}
        \STATE $h_N\leftarrow\frac{\varepsilon}{||\nabla g_{m(N)}(x^N)||_{*}^2}$
        \STATE $x^{N+1}\leftarrow Mirr_{x^N}(h_N\nabla g_{m(N)}(x^N))\;\text{// \emph{"непродуктивные шаги"}}$
    \ENDIF
    \STATE $N\leftarrow N+1$
\UNTIL{$\Theta_0^2 \leqslant \frac{\varepsilon^2}{2}\left(|I|+\sum\limits_{k\not\in I}\frac{1}{||\nabla g_{m(k)}(x^k)||_{*}^2}\right)$}
\ENSURE $\bar{x}^N:=argmin_{x^k,\;k\in I}\,f(x^k)$
\end{algorithmic}
\end{algorithm}

Для оценки скорости сходимости этого метода подобно (\cite{bib_Nesterov}, п. 3.2.2), для всякого ненулевого конечного субградиента (элемента субдифференциала Кларка) $\nabla f(x)$ целевого квазивыпуклого функционала $f$ введём следующую вспомогательную величину
\begin{equation}\label{eq:vfDef}
v_f(x, y) = \left\langle\frac{\nabla f(x)}{\|\nabla f(x)\|_{*}},x-y\right\rangle,\quad x \in Q.
\end{equation}

Аналогично (\cite{bib_Stonyakin}, теорема 2) с использованием леммы \ref{lem1} проверяется следующая
\begin{theorem}\label{th2}
Пусть $\varepsilon > 0$~--- фиксированное число и выполнен критерий остановки алгоритма \ref{alg4}. Тогда
\begin{equation}\label{eq09}
\min\limits_{k \in I} v_f(x^k,x_*)<\varepsilon.
\end{equation}
Отметим, что алгоритм \ref{alg4} работает не более
\begin{equation}\label{eqq08}
N=\left\lceil\frac{2\max\{1, M_g^2\}\Theta_0^2}{\varepsilon^2}\right\rceil
\end{equation}
итераций.
\end{theorem}

Теперь покажем, как можно оценить скорость сходимости предлагаемого метода. Для этого полезно следующее вспомогательное утверждение, которое есть аналог (\cite{bib_Nesterov}, лемма 3.2.1). Напомним, что под $x_*$ мы понимаем точное решение задачи \eqref{equv2}--\eqref{problem_statement_g}. Отличительной особенностью данного утверждения является то, что мы рассматриваем не выпуклый, а квазивыпуклый целевой функционал $f$. Предположение о его локальной липшицевости позволяет в качестве аппарата для исследования дифференциальных свойств использовать субдифференциал Кларка.
\begin{theorem}\label{Nesterov_lem}
Пусть $f: Q \rightarrow \mathbb{R}^n$ --- локально липшицев квазивыпуклый функционал.
Введем следующую функцию:
\begin{equation}\label{eq13}
\omega(\tau)=\max\limits_{x\in Q}\{f(x)-f(x_*):||x-x_*||\leqslant\tau\},
\end{equation} где $\tau$ - положительное число.
Тогда для всякого $x \in Q$
\begin{equation}\label{eq_lemma}
f(x) - f(x_*) \leqslant \omega(v_f(x,x_*)).
\end{equation}
\end{theorem}
\begin{proof}
Мы отправляемся от схемы рассуждений (\cite{bib_Nesterov}, лемма 3.2.1) с тем отличием, что вместо обычного субдифференциала выпуклой функции будет использоваться субдифференциал Кларка. Можно проверить, что $$v_f(x, x_*)=\min\limits_y\{||y-x_*||: \langle\nabla f(x),y-x\rangle=0\}.$$
Действительно, пусть $v_f(x, x_*) = ||y_*-x_*||$ для некоторого $y_*$: $\langle\nabla f(x), y_*-x\rangle=0$. Тогда $\nabla f(x)=\lambda s$, где $\langle s,y_*-x_*\rangle=||y_*-x_*||$ для некоторого $s$ такого, что $||s||_*=1$.
Поэтому
$$0=\langle\nabla f(x),y_*-x\rangle= \lambda\langle s,y_*-x_*\rangle+\langle\nabla f(x),x_*-x\rangle,$$
откуда
$$\lambda=\frac{\langle\nabla f(x),x-x_*\rangle}{||y_*-x_*||}=||\nabla f(x)||_* \text{  и  } v_f(x, x_*)= ||y_*-x_*||.$$

Остаётся лишь учесть существование конечной производной по направлению $h \in Q$ у всякого локально липшицева квазивыпуклого функционала $f$
$$
f'(x, h) = \lim\limits_{\lambda \downarrow 0} \frac{f(x+\lambda h) - f(x)}{\lambda}.
$$

Далее, с использованием свойства квазивыпуклости и \eqref{equv_Clarke} для локально липшицева квазивыпуклого функционала получаем:
$$
f'(x, h) = f_{Cl}^{\uparrow} (x, h) = \max\limits_{\nabla f(x) \in \partial_{Cl}f(x)} \langle \nabla f(x), h \rangle.
$$
Для всякого направления $h$ такого, что $\langle \nabla f(x), h \rangle > 0$ получаем $f'(x, h) > 0$. Поэтому имеет место
$f(x + \lambda h) \geqslant f(x)$ для произвольного направления $h$ такого, что $\langle \nabla f(x), h \rangle > 0$.
Неравенство $f(y) - f(x) \geqslant 0$ следует из теперь непрерывности функционала $f$ для всякого $y$ такого, что $\langle\nabla f(x),y-x\rangle=0$. Итак,
$$
f(x) - f(x_*) \leqslant f(y) - f(x_*) \leqslant \omega(v_f(x,x_*)).
$$
\end{proof}

На базе теорем \ref{th2} и \ref{Nesterov_lem} можно оценить скорость сходимости алгоритма для квазивыпуклого локально липшицева целевого функционала $f$ с липшицевым субградиентом. Используя доказанное в теореме \ref{th1} неравенство
$$f(x)\leqslant f(x_*)+(||\nabla f(x_*)||_* + \delta)||x-x_*||+\frac{1}{2}L||x-x_*||^2,$$
мы можем получить, что
$$\min\limits_{k\in I}f(x^k)-f(x_*)\leqslant\min\limits_{k\in I} \left\{(||\nabla f(x_*)||_*+ \delta)||x^k-x_*||+\frac{1}{2}L||x^k-x_*||^2\right\}.$$
Далее, по теореме \ref{Nesterov_lem} верно неравенство:
$$
f(x)-f(x_*)\leqslant \varepsilon \cdot (||\nabla f(x_*)||_*+ \delta) + \frac{1}{2} L\varepsilon^2.
$$
Поэтому справедливо

\begin{corollary}\label{cor1}
Пусть локально липшицев квазивыпуклый функционал $f$ имеет липшицев субградиент. Тогда после остановки алгоритма  верна оценка:
\begin{equation}\label{sol1}
\min\limits_{1\leqslant k\leqslant N}f(x^k)-f(x_*)\leqslant \varepsilon \cdot (||\nabla f(x_*)||_* + \delta)+\frac{L\varepsilon^2}{2},
\end{equation}
причём для всякого $k$
\begin{equation}\label{sol2}
g_m(x^k) \leq \varepsilon \quad \forall m = \overline{1, M}.
\end{equation}
\end{corollary}

Таким образом, остановка алгоритма 1 гарантирует достижение приемлемого качества найденного решения \eqref{sol1} --- \eqref{sol2}, а оценка \eqref{eqq08} указывает на его оптимальность с точки зрения нижних оракульных оценок \cite{nemirovsky1983problem} даже в классе выпуклых (а тем более и квазивыпуклых) целевых функционалов.

Полученные результаты, в частности, позволяют сделать такие выводы. Во-первых, алгоритм 1 применим для задач минимизации не только выпуклых, но и квазивыпуклых целевых функционалов. Во-вторых, особенности поведения целевого функционала в окрестности некоторых отдельных точек могут не сильно усложнять интерполяцию (модель) оптимизируемой функции, что может позволить сохранять при наличии таких особенностей оценки скорости сходимости метода.

\subsection*{Благодарности} Автор выражает огромную признательность Александру Владимировичу Гасникову и Юрию Евгеньевичу Нестерову за полезные обсуждения и рекомендации.


\begin{center}
Статья направлена в журнал "Труды Института математики и механики Уральского отделения РАН" 10.12.2018
\end{center}

\end{document}